\documentclass{amsart}
\usepackage{amsmath, amssymb, amsthm}
\newtheorem{thr}{Theorem}
\newtheorem{lem}[thr]{Lemma}

\newtheorem{stat}[thr]{Proposition}

\theoremstyle{definition}
\newtheorem{defn}[thr]{Definition}
\newtheorem{example}[thr]{Example}
\newtheorem{question}[thr]{Question}

\theoremstyle{remark}
\newtheorem{remr}[thr]{Remark}

\numberwithin{equation}{section}

\def\R{\mathbb{R}}
\def\Ro{\overline{\mathbb{R}}}
\def\t{^{\tau}}

\def\S{\mathbb{S}}
\def\F{\mathcal{F}}
\def\per{\operatorname{per}}
\def\T{\mathcal{T}}

\begin{document}

\title{A separation between tropical matrix ranks}

\author{Yaroslav Shitov}
\email{yaroslav-shitov@yandex.ru}






\begin{abstract}
We continue to study the rank functions of tropical matrices. In this paper, we explain how to reduce the computation of ranks for matrices over the `supertropical semifield' to the standard tropical case. Using a counting approach, we prove the existence of a $01$-matrix with many ones and without large all-one submatrices, and we put our results together and construct an $n\times n$ matrix with tropical rank $o(n^{0.5+\varepsilon})$ and Kapranov rank $n-o(n)$.
\end{abstract}

\maketitle




The \textit{tropical} arithmetic operations on $\R$ are $(a,b)\to\min\{a,b\}$ and $(a,b)\to a+b$. One can complete this algebraic structure with an element neutral with respect to addition (it plays the role of an infinite positive element and is denoted by $\infty$) and get the structure $(\Ro,\min,+)$ known as the \textit{tropical semiring}. This semiring and related structures are being studied since the 1960's because of their applications in optimization theory~\cite{Vor}; the tropical methods also arise naturally in algebraic geometry and lead to important developments in the field (see e.g.~\cite{BH, CDPR, Mikh}). Other applications of tropical mathematics include operations research~\cite{CG}, discrete event systems~\cite{BCOQ}, automata theory~\cite{Sim}, and optimal control~\cite{KM, McE}.

This paper is a continuation of the study of tropical rank functions initiated in~\cite{DSS} and developed in~\cite{CJR, mylaa, mytrb, myproc}. Namely, we are going to focus on the tropical rank and Kapranov rank---the functions arisen from the context of tropical algebraic geometry. (It should be mentioned that there are many more rank functions of tropical matrices that are being extensively studied in the literature, see~\cite{AGG, GS} and references therein.) We refer the reader to~\cite{CJR, DSS, mytrb} for definitions and a detailed discussion of motivation behind these concepts; the rank functions also admit combinatorial descriptions, and we are going to recall them in Section~2. For the purpose of this introduction, we briefly recall that the \textit{tropical rank} of a matrix is the topological dimension of the tropical convex hull of its columns, and the \textit{Kapranov rank} is the smallest dimension of tropical linear spaces containing these columns. As we will see, one needs to specify a field $\mathbb{F}$ to give the definition of Kapranov rank, and the corresponding function is referred to as the Kapranov rank \textit{over} $\mathbb{F}$.
 
It is very well known that these functions can be different, but the tropical rank cannot exceed any of the Kapranov ranks (see~\cite{DSS}). The papers~\cite{CJR, DSS, mytrb} contain a resolution of the following question: For which $d,n,r$ does every $d\times n$ matrix of tropical rank less than $r$ have Kapranov rank (over $\mathbb{C}$) less than $r$ as well? This condition is equivalent to the $r\times r$ minors of a $d\times n$ matrix of variables being what is called a \textit{tropical basis} of the ideal that they generate. The paper~\cite{DSS} contains an example showing that the $4\times4$ minors of a $7\times 7$ matrix are not a tropical basis, and it was asked if it is the case for the minors of the $5\times 5$ matrix. The authors of~\cite{CJR} proved that the $4\times 4$ minors of a $5\times 5$ matrix form a tropical basis, and a complete description of the tuples $(d,n,r)$ for which the answer to the above question is positive was given in~\cite{mytrb}. The result below is valid for the Kapranov rank function computed with respect to any infinite field.

\begin{thr}\label{thrmytrb}
Let $d,n,r$ be positive integers, $r\leq\min\{d,n\}$. The $d$-by-$n$ matrices with tropical rank less than $r$ always have Kapranov rank less than $r$ if and only if one of the following conditions holds:

(1) $r\leq3$;

(2) $r=\min\{d,n\}$;

(3) $r=4$ and $\min\{d,n\}\leq6$.
\end{thr}

The answer to the same question but for finite fields remains unknown, but the corresponding characterization should be different from the one given in Theorem~\ref{thrmytrb}. In fact, Example~2.7 in~\cite{mylaa} shows that matrices of tropical rank two and larger Kapranov rank exist for any finite ground field. As we see, there are a lot of results describing the cases when the tropical and Kapranov ranks are equal; on the opposite end, there is a result (see~\cite{KR}) stating that a matrix of tropical rank three can have arbitrarily large Kapranov rank. No non-trivial analogue of this result is known if we compare the behavior of the Kapranov rank functions taken over different fields.

\begin{question}(See also Question~5 in Section~8 of~\cite{DSS} and Problem~4.1 in~\cite{LGAG}.)
For which $d$ does there exist a matrix with rational Kapranov rank three and real Kapranov rank $d$? Clearly, this may be possible only if $d\geqslant 3$, and for $d=3$ the answer is trivially positive. The answer is also positive for $d=4$, and to see this, one can construct the \textit{cocircuit matrix} as in~\cite{DSS} for the matroid corresponding to the \textit{Perles configuration} (see page~94 of~\cite{Grun}). Nothing is known for $d\geqslant 5$.  
\end{question}

We note that the difference between the tropical rank and Kapranov rank in the above mentioned example is of the order of $0.5\sqrt{n}$, for an $n\times n$ tropical matrix (see Theorem~2.4 in~\cite{KR}). One can also construct a sequence of $n\times n$ matrices whose tropical rank and Kapranov rank differ by $\varepsilon n$ as a diagonal matrix whose diagonal blocks are equal to any fixed matrix with different ranks. In our paper, we improve these bounds to the asymptotically best possible separation of $n-o(n)$; our result is valid for the Kapranov ranks over all fields.

\begin{thr}\label{thrmainthis}
For all $n>1000$, $\alpha\in(0,0.1)$, there is an $n\times n$ matrix $A$ such that
$$\operatorname{tropical\,\,rank\,\,}(A)\leqslant\frac{4\sqrt{n}\ln n}{\alpha^2}\,\,\,\,\mbox{and}\,\,\,\,\operatorname{Kapranov\,\,rank\,\,}(A)\geqslant n(1-\alpha).$$
\end{thr}

Taking $n\to\infty$ and choosing $\alpha_n$ to be $(\ln n)^{-1}$ or any other sequence sufficiently slowly decreasing to $0$, we get an $n-o(n)$ separation between the tropical rank and Kapranov rank of an $n\times n$ tropical matrix. Since the Kapranov rank is a lower bound for the tropical factorization rank (see~\cite{DSS}), Theorem~\ref{thrmainthis} gives an $n-o(n)$ separation between the tropical rank and factorization rank as well. A similar question is wide open for separations between the conventional rank and factorization rank of nonnegative matrices, and related problems have important applications in optimization and computational complexity theory (see~\cite{FMPTdW}). The following version of this problem is open in both the tropical and nonnegative settings.

\begin{question}
Let $k$ be a fixed constant. Is it correct that, for all $n,m$ satisfying $m\geqslant n\geqslant k$, there exists an $n\times m$ matrix with tropical rank $k$ and tropical factorization rank $n$? Does there exist a nonnegative $n\times m$ matrix with conventional rank $k$ and nonnegative rank $n$?
\end{question}

The `nonnegative part' of this question has negative answer for $k\leqslant 3$ (this is easy for $k\leqslant 2$ as shown in~\cite{CR}, and the case of $k=3$ has been done in~\cite{PP, my7x7}). If $k\geqslant 4$, this problem remains open, see Question~1 in~\cite{Hru}. The `tropical part' is open already for $k=3$, and the problem is non-trivial even in the case $k\leqslant 2$ for which Theorem~4.6 in~\cite{myspb} gives a negative answer.


\section{Preliminaries. The rank of a supertropical matrix}

This paper was inspired by the idea of \textit{symmetrized semirings} (see~\cite{AGG, Plus}), which are intended to give an analogue of subtraction for those semirings that are not rings. The \textit{symmetrized tropical semiring} is essentially the set $\Ro\times\Ro$, and we may think of a pair $(r_1,r_2)$ as a formal subtraction $r_1\ominus r_2$. The tropical operations can be naturally extended to the symmetrized setting as $(r_1,r_2)\oplus(s_1,s_2)=(\min\{r_1,r_2\},\min\{s_1,s_2\})$ and $(r_1,r_2)\odot(s_1,s_2)=(\min\{r_1+s_1,r_2+s_2\},\min\{r_1+s_2, r_2+s_1\})$ because $\min$ is the tropical addition and $+$ is the tropical multiplication. A related structure was introduced by Izhakian and Rowen in~\cite{Izh, IRIR1} and became known as the `\textit{supertropical semifield}'. Their structure `\textit{is somehow reminiscent of the symmetrized max-plus semiring, and has two kind of elements, the “real” ones (which can be identified to elements of the max-plus semiring and some “ghost” elements which are similar to the “balanced” ones,}' as Akian, Gaubert, and Guterman wrote in~\cite{AGG}. We decided to write this paper in terms of the `\textit{supertropical}' structure because it seems to have become more popular nowadays due to a considerable amount of papers on the topic written by Izhakian, Rowen, and their colleagues (see also~\cite{IKR, IKR2} and references therein).

As said above, the structure introduced by Izhakian and Rowen belongs to the class most commonly known as `\textit{supertropical semifields}' (see~\cite{IKR}), but since it contains non-zero elements without multiplicative inverses, it is not an actual semifield according to the standard definition of the latter. We denote this structure by $\S=(\R^{\tau}\cup\R^{\gamma}\cup\{\infty\},\oplus,\odot)$, where $\infty$ is an infinite positive element, and $\R^\tau$, $\R^\gamma$ are two copies of $\R$ whose elements are called in the literature `\textit{tangible}' and `\textit{ghost}', respectively. Assuming $i,j\in\{\tau,\gamma\}$, $s\in\S$, $a,b\in \R$ and $a>b$, we define the operations by

\noindent $\infty\oplus s=s\oplus \infty=s$, $\infty\odot s=s\odot \infty=\infty$;

\noindent $b^{j}\oplus a^{i}=a^{i}\oplus b^{j}=b^{j}$, $b^{i}\oplus b^{j}=b^\gamma$;

\noindent $a^{i}\odot b^{j}=(a+b)^{\alpha}$,

\noindent where $\alpha=\tau$ if $i=j=\tau$, and $\alpha=\gamma$ otherwise. One can check that $\oplus$ and $\odot$ are commutative and associative operations, and distributivity also holds. Moreover, there is a homeomorphism $\nu$ from $\S$ to the tropical semiring $\Ro$ defined by $\infty\to\infty$ and $a^{i}\to a$. One writes $c\models d$ if either $c=d$ or $c=d\oplus g$, for some ghost element $g$; this relation is known in the literature as `\textit{ghost surpassing}' relation.

Let $A=(a_{ij})$ be an $n\times n$ supertropical matrix; its \textit{permanent} is
$$\operatorname{per} A=\bigoplus_{\sigma\in S_n}A(1|\sigma_1)\odot\ldots\odot A(n|\sigma_n),$$
where $S_n$ denotes the symmetric group on $\{1,\ldots,n\}$ and $A(p|q)$ denotes the entry in the $p$th row and $q$th column of $A$. This matrix is said to be \textit{tropically non-singular} if $\operatorname{per} A$ is tangible and \textit{tropically singular} otherwise.

\begin{defn}\label{defntr}
The \textit{tropical rank} of a supertropical matrix is the largest size of its non-singular square submatrix.
\end{defn}

In order to recall the definition of \textit{Kapranov rank}, we need to introduce the field $\F=\mathbb{F}\{\{t\}\}$ of \textit{generalized Puiseux series}. The elements of $\F$ are formal sums $a(t)=\sum_{e\in\mathbb{R}} a_et^{e}$ which have coefficients $a_e$ in $\mathbb{F}$ and whose \textit{support} $\operatorname{Supp}(a)=\{e\in\mathbb{R}: a_e\neq0\}$ is well-ordered (which means that every non-empty subset of $\operatorname{Supp}(a)$ has a minimal element). The \textit{tropicalization} mapping $\deg:K\rightarrow\Ro$ sends a series $a$ to the exponent of its \textit{leading term}; in other words, we define $\deg a=\min\operatorname{Supp}(a)$ and $\deg 0=\infty$.

\begin{defn}\label{defKap}
The \textit{Kapranov rank} of a supertropical matrix $A$ is the smallest possible rank of a matrix $L$ whose entries are in $\F$ and which satisfies $A\models\deg L$. (Such a matrix $L$ is to be called a \textit{lifting} of $A$.)
\end{defn}

\begin{remr}\label{remr1}
If $A$ is a supertropical matrix without ghost elements, then these rank functions match those of conventional tropical matrices as introduced and studied in~\cite{CJR, DSS, mylaa, mytrb}. Namely, the tropical rank and Kapranov rank of a matrix $T$ with entries in $\Ro$ coincide with those in Definitions~\ref{defntr} and~\ref{defKap} if the elements in $\R$ are replaced by their tangible copies.
\end{remr}

We finalize the section by proving two results using well known techniques.

\begin{stat}\label{statexp}
Let $A$ be an $n\times n$ supertropical matrix satisfying $A_{11}=A_{21}=0\t$ and $A_{31}=\ldots=A_{n1}=\infty$. Let $B$ be the $(n-1)\times(n-1)$ matrix obtained from $A$ by removing the first column and replacing the first two rows by their (supertropical) sum. Then $\operatorname{per} A=\operatorname{per} B$.
\end{stat}

\begin{proof}
Let us expand $\operatorname{per} A$ with the first column (as in Lemma~3.2 in~\cite{PH}). Since $\infty$ and $0\t$ are neutral with respect to $\oplus$ and $\odot$, respectively, we get that $\operatorname{per} A$ is the sum of the permanents of the $(1,1)$ and $(2,1)$ cofactors of $A$. Since the permanent of a matrix is linear in its rows (see Lemma~3.13 in~\cite{PH}), the result follows.
\end{proof}

\begin{lem}\label{lemeas}
Let $A$ be a non-singular supertropical $n\times n$ matrix and $a,b\in\R$. Then one of the columns of $A$ can be replaced by $v=(a\t\,b\t\,\infty\ldots\infty)^\top$ so that the resulting matrix remains non-singular. Moreover, we can choose a column in which one of the first two entries is tangible.
\end{lem}

\begin{proof}
Since permutations of rows and columns and their scaling by elements in $\R^\tau$ cannot affect the non-singularity, we can apply the \textit{Hungarian algorithm} for the assignment problem corresponding to the matrix $\nu(A)$, see~\cite{Kuhn} for details. Therefore, we can assume without loss of generality that the diagonal entries of $A$ are equal to $0^\tau$, and the off-diagonal entries are positive. If $a<b$, then we choose the first column to be replaced by $v$, and otherwise we replace the second column.
\end{proof}

\section{Reducing the supertropical case to tropical matrices}\label{secred}

Let $S\in\S^{I\times J}$ be a supertropical matrix, where $I$ and $J$ denote the row and column indexing sets which we assume to be disjoint. Let us denote by $I^1$, $I^2$ two copies of the set $I$, and $i^1, i^2$ will stand for the elements that correspond to $i\in I$ in these copies.

\begin{defn}\label{def1111}
We call a tropical matrix $T$ \textit{symmetrized} if it has $I^1\cup I^2$ as row indexing set and $I\cup J$ as column indexing set (where $I,I^1, I^2, J$ are as above), and satisfies $T(i^1|i)=T(i^2|i)=0$, $T(i^1|\hat{\iota})=T(i^2|\hat{\iota})=\infty$ for all $i\in I$, $\hat{\iota}\in I\setminus\{i\}$.
\end{defn}

\begin{defn}\label{def11111}
Let $T$ be a matrix as in the above definition. We define $\Sigma\in\S^{I\times J}$ as the matrix obtained from $T$ by putting on the $i$th place the (supertropical) sum of the $i^1$th and $i^2$th rows of $T$ and removing the columns with indexes in $I$.
\end{defn}

The relation between $T$ and $\Sigma(T)$ can be understood in terms of the symmetrized tropical semiring discussed above. As we see, $T$ is obtained from $\Sigma(T)$ by replacing every row with a pair of rows which can be thought of as a vector over the symmetrized semiring corresponding to the row of $T$. Let us illustrate this construction with the example essentially appeared in the paper~\cite{myarx1} published in 2010. (Namely, the matrix $\T(A)$ below is the one from Example~2.1 in~\cite{myarx1} up to permutations of rows and columns and replacing the $4$'s by the $\infty$'s, which is not crucial for the argument given in~\cite{myarx1}.)

\begin{example}\label{exam1111}
Consider the symmetrized tropical matrix
$$A=\begin{pmatrix}
0&\infty&\infty&0&2&1\\
\infty&0&\infty&2&0&2\\
\infty&\infty&0&2&1&0\\
0&\infty&\infty&0&\infty&\infty\\
\infty&0&\infty&\infty&0&\infty\\
\infty&\infty&0&\infty&\infty&0
\end{pmatrix}
$$
and its supertropical counterpart
$$\Sigma(A)=\begin{pmatrix}
0^\gamma&2^\tau&1^\tau\\
2^\tau&0^\gamma&2^\tau\\
2^\tau&1^\tau&0^\gamma
\end{pmatrix}
$$
constructed as in Definition~\ref{def11111}. If the rows of $\Sigma(A)$ had indexes $1,2,3$ (from top to bottom) and its columns had indexes $4,5,6$ (from left to right), then the rows of $A$ are indexed with $1^1,2^1,3^1,1^2,2^2,3^2$, and the columns of $A$ with $1,2,3,4,5,6$.
\end{example}

One can check it directly that the tropical rank and Kapranov rank of the matrix $\Sigma(A)$ above are $1$ and $2$, respectively; it is proven in~\cite{myarx1} that the respective ranks of $A$ are $4$ and $5$. We can begin proving the main results of this section, which give a general relation between the ranks of $A$ and $\Sigma(A)$.

\begin{thr}\label{equalKap}
Let $T$ be a matrix as in Definition~\ref{def1111}. If $|\mathbb{F}|\geqslant3$, then
$$\operatorname{Kapranov\,\, rank\,\,}T=\operatorname{Kapranov\,\, rank\,\,}\Sigma(T)+|I|.$$
\end{thr}

\begin{proof}
Let $\mathcal{L}$ be a lifting of $T$ with smallest possible rank $r$; row scalings allow us to assume that $\mathcal{L}(i^1|i)=\mathcal{L}(i^2|i)=1$ for all $i\in I$. Let $L'$ be the matrix obtained from $\mathcal{L}$ by subtracting, for any $i$, the $i^1$th row from $i^2$th row. The ranks of $\mathcal{L}$ and $L'$ are equal, and we have
\begin{equation}\label{eqLLLL}L'=\left(\begin{array}{c|c}
\mathcal{I}&*\\\hline
\mathcal{O}&L
\end{array}\right),
\end{equation}
where $\mathcal{I}$ and $\mathcal{O}$ are, respectively the unit and zero $|I|\times|I|$ matrices, $L$ is a lifting of $\Sigma(T)$, and $*$ stands for a matrix we need not specify. Therefore, the Kapranov rank of $\Sigma(T)$ is at most $r-|I|$.

Conversely, consider a lifting $L$ of $\Sigma(T)$ with smallest possible rank $\rho$. We define the matrix $\mathcal{L}\in\F^{(I^1\cup I^2)\times(I\cup J)}$ as follows. For all $i\in I$, $j\in J$, $\hat{\iota}\in I\setminus\{i\}$, we set

\noindent (i) $\mathcal{L}(i^1|i)=\mathcal{L}(i^2|i)=1$ and $\mathcal{L}(i^1|\hat{\iota})=\mathcal{L}(i^2|\hat{\iota})=0$,

\noindent (ii) $\mathcal{L}(i^1|j)=\zeta t^s$, $\mathcal{L}(i^2|j)=L(i|j)+\zeta t^s$ if $T(i^1|j)=T(i^2|j)=s$,

\noindent (iii) $\mathcal{L}(i^1|j)=t^s$, $\mathcal{L}(i^2|j)=L(i|j)+t^s$ if $s=T(i^1|j)>T(i^2|j)$,

\noindent (iv) $\mathcal{L}(i^1|j)=t^s-L(i|j)$, $\mathcal{L}(i^2|j)=t^s$ if $T(i^1|j)<T(i^2|j)=s$.

Since the ground field $\mathbb{F}$ contains more than two elements, we can avoid cancellation of leading (degree-$s$) terms in $L(i|j)+\zeta t^s$ by choosing an appropriate non-zero value of $\zeta$ in $\mathbb{F}$. As we see, the constructed matrix $\mathcal{L}$ is a lifting of $T$, and in order to compute its rank we subtract, as above, the $i^1$th row from $i^2$th row, for any $i$. We get the matrix $L'$ as in~\eqref{eqLLLL}, so the rank of $\mathcal{L}$ equals $\rho+|I|$.
\end{proof}

\begin{thr}\label{equaltrop}
Let $T$ be a matrix as in Definition~\ref{def1111}. Then
$$\operatorname{tropical\,\, rank\,\,}T=\operatorname{tropical\,\, rank\,\,}\Sigma(T)+|I|.$$
\end{thr}

\begin{proof}
Denote by $I_0,J_0$ the sets of row and column indexes of a largest non-singular submatrix $C$ of $\Sigma(T)$. Denoting $|I|=n$ and $|I_0|=|J_0|=r$, we observe that the submatrix of $T$ formed by the rows with indexes in $I_0^2\cup I^1$ and columns with indexes in $I\cup J_0$ looks like (with upper left block having row indexes in $I_0^2\cup I_0^1$ and column indexes in $I\setminus I_0$)
\begin{equation}\label{eqTTTT}
\left(\begin{array}{c|c}
\mathcal{Z}_{2r\times(n-r)}&T'\\\hline
\mathcal{U}_{n-r}&*
\end{array}\right),
\end{equation}
where $\mathcal{U}$ is the tropical unit matrix (the one with $0$'s on the diagonal and $\infty$'s everywhere else), $\mathcal{Z}$ is the all-$\infty$ matrix, and $T'$ is a symmetrized matrix such that $\Sigma(T')=C$. The permanent of~\eqref{eqTTTT} equals $\per(T')$, which in turn, according to Proposition~\ref{statexp}, equals $\per(C)$. Since $C$ is non-singular, it has a tangible permanent, and so does the matrix~\eqref{eqTTTT}, which is therefore non-singular as well. In particular, we get a `$\geqslant$' inequality for the values in the formulation of the lemma.

In order to prove the `$\leqslant$' inequality, we use Lemma~\ref{lemeas} and observe that any of the largest tropically non-singular submatrices of $T$ can be reduced to the form~\eqref{eqTTTT}, and the proof can be finalized as in the previous paragraph.
\end{proof}

\section{Constructing a tropical matrix}

In this section, we explain how to construct a tropical matrix with small tropical rank and large Kapranov rank if we are given a $0-1$ matrix as below.

\begin{defn}\label{defgood}
Numbers $(d,k,r,u)$ are said to be a \textit{good} tuple if there exists an $d\times (kd-d)$ matrix $\mathcal{M}$ of zeros and ones such that

\noindent (1) at least $u$ entries of $\mathcal{M}$ are ones;

\noindent (2) any $\rho\times\rho$ submatrix of $\mathcal{M}$ contains a zero unless $\rho<r$.
\end{defn}

We enumerate the rows and columns of $\mathcal{M}$ by disjoint sets $I$ and $J$ and construct the tropical matrix $\Phi=\Phi(\mathcal{M})$ as follows. Its rows are indexed with $\{1,\ldots,k\}\times I$, its columns with $I\cup J$, and its entries are

\noindent (1) $\Phi(\alpha,i|i)=0$ and $\Phi(\alpha,i|\hat{\imath})=\infty$ if $\hat{\imath}\in I\setminus\{i\}$;

\noindent (2) $\Phi(\alpha,i|j)=0$ if $j\in J$ and $\mathcal{M}(i|j)=0$;

\noindent (3) $\Phi(\alpha,i|j)=a_{ij\alpha}$ if $j\in J$ and $\mathcal{M}(i|j)=1$,

\noindent where $(a_{ij\alpha})$ are a family of numbers in $[1,1+1/(kd)]$ that are linearly independent over $\mathbb{Q}$. Notice that $\Phi$ is an $n\times n$ matrix with $n=kd$.

\begin{lem}\label{lemkapb}
$\operatorname{Kapranov\,\, rank\,\,}\Phi\geqslant n-\sqrt{n^2-ku}$.
\end{lem}

\begin{proof}
Any lifting of $\Phi$ has $ku$ entries with degrees in $(a_{ij\alpha})$, and since these degrees are linearly independent over $\mathbb{Q}$, the corresponding entries should be algebraically independent over $\mathbb{F}$. It remains to note that any matrix of rank $\rho$ has transcendence degree at most $2n\rho-\rho^2$ and resolve the inequality $2n\rho-\rho^2\geqslant ku$ for $\rho$.
\end{proof}

\begin{lem}\label{lemtropb}
$\operatorname{Tropical\,\, rank\,\,}\Phi\leqslant d+kr$.
\end{lem}

\begin{proof}
Let $H$ be a square submatrix of $\Phi$ of size greater than $d+kr$. We need to check that $H$ is singular, that is, that the permanent of $H$ seen as a supertropical matrix is either $\infty$ or a ghost. By Dirichlet's principle, there is a subset $\mathcal{I}\subset I$ of cardinality $\rho\geqslant r$ such that, for any $i\in\mathcal{I}$, there are two distinct pairs $(\alpha_i,i)$ and $(\beta_i,i)$ which appear to be row indexes of $H$. We denote by $H_0$ the submatrix of $\Phi$ formed by the rows with indexes in $\bigcup_{i\in\mathcal{I}}\{(\alpha_i,i),(\beta_i,i)\}$; we will be done if we manage to show that every $2\rho\times 2\rho$ submatrix of $H_0$ is singular.

By Theorem~\ref{equaltrop}, we need to show that every $\rho\times \rho$ submatrix of $\Sigma(H_0)$ is singular. Removing the columns of $\Sigma(H_0)$ consisting of $\infty$-entries, we get a matrix $\mathcal{H}\in\S^{\mathcal{I}\times J}$ such that 

\noindent (1) $\mathcal{H}(i|j)=0^\gamma$ if $\mathcal{M}(i|j)=0$, and

\noindent (2) $\mathcal{H}(i|j)=a_{ij}^\tau$ with $a_{ij}\in[1,1+1/n]$ if $\mathcal{M}(i|j)=1$.

Since every $\rho\times \rho$ submatrix of $\mathcal{M}$ contains a zero, the permanent of every $\rho\times \rho$ submatrix of $\mathcal{H}$ should have a summand $g^\gamma$ with $g\leqslant 0+(r-1)(1+1/n)<r$. Therefore, the products of tangible entries do not contribute to the permanent of any $r\times r$ submatrix of $\mathcal{H}$.
\end{proof}

\begin{remr}
Many authors (see~\cite{CJR, DSS}) consider tropical matrices with finite entries only. We note that the bounds as in Lemmas~\ref{lemkapb} and~\ref{lemtropb} will still hold if we replace every $\infty$ in the entries of $\Phi$ by $2$. In fact, the resulting matrix can be obtained as $D\odot\Phi$, where $D$ is the $n\times n$ matrix with $0$'s on the diagonal and $2$'s everywhere else. Of course, the tropical rank of $D\odot\Phi$ is at most that of $\Phi$ (see Theorem~9.4 in~\cite{AGG}), and the proof of Lemma~\ref{lemkapb} reads equally well if we replace $\Phi$ by $D\odot\Phi$.
\end{remr}

\section{Constructing the matrix $\mathcal{M}$}

In this section, we give a probabilistic construction of the matrix as in Definition~\ref{defgood} and finalize the proof of Theorem~\ref{thrmainthis}.

\begin{lem}\label{lemgood}
Let $q\in(0,0.1)$ and $d\geqslant 2$. Then the numbers $$k=d,\mbox{$ $ $ $ $ $ $ $}r=4\ln d/q,\mbox{$ $ $ $ $ $ $ $}u=(1-q-d^{-1.5})(d^3-d^2)$$ are a good tuple in the sense of Definition~\ref{defgood}.
\end{lem}

\begin{proof}
Let $X$ be a random $d\times(d^2-d)$ matrix with independent entries each of which is either $0$ or $1$, and the probability of $0$ is $q$. According to Hoeffding's inequality (see Theorem~1 in~\cite{Hoeff}), the probability that the number of $1$-entries of $X$ does not exceed $u$ is at most $\exp(-2d^{-3}(d^3-d^2))<0.5$. Therefore, the condition~(1) as in Definition~\ref{defgood} fails with probability less than $0.5$.

We proceed with condition~(2), whose negation means that $X$ has an $\lceil r\rceil\times\lceil r\rceil$ submatrix of all ones. The probability that this happens with any particular such submatrix is at most $(1-q)^{r^2}$, and the number of these submatrices does not exceed $(d^2-d)^{r+1}d^{r+1}$. Therefore, the condition~(2) fails with probability at most
$$d^{3r+3}(1-q)^{r^2}<e^{\frac{12(\ln d)^2}{q}+3\ln d+\frac{16(\ln d)^2\ln (1-q)}{q^2}}<e^{\frac{(\ln d)^2}{q}\left(15+16\frac{\ln (1-q)}{q}\right)},$$
which is also less than $0.5$ because $\ln (1-q)/q<-1$.
\end{proof}

Now we can complete the proof of Theorem~\ref{thrmainthis}.

\begin{proof}[Proof of Theorem~\ref{thrmainthis}.] We define $d=\lfloor\sqrt{n}\rfloor$ and write 
$q=\left(\alpha-2n^{-0.25}\right)^2$. We can assume without loss of generality that the bound for the tropical rank is less than $n$ because otherwise the result is trivial. In particular, we have that $\alpha>2n^{-0.25}\sqrt{\ln n}$.

The numbers $d,k,r,u$ as in Lemma~\ref{lemgood} allow us to construct a $d^2\times d^2$ matrix $\Phi$ satisfying the assumptions of Definition~\ref{defgood}; we complete $\Phi$ to an $n\times n$ matrix $\Phi_0$ by adding the copies of existing rows and columns. According to Lemma~\ref{lemkapb}, the Kapranov rank of $\Phi_0$ is at least
$$d^2-\sqrt{d^4-d^4(1-d^{-1})(1-q-d^{-1.5})}\geqslant d^2\left(1-\sqrt{q}-\sqrt{d^{-1}+d^{-1.5}}\right),$$
which is greater than or equal to $n(1-\sqrt{q}-2n^{-0.25})=n(1-\alpha)$. By Lemma~\ref{lemtropb}, the tropical rank of $\Phi_0$ does not exceed
$$\frac{4d\ln d}{q}+d\leqslant\frac{2.4\sqrt{n}\ln n}{q}\leqslant\frac{4\sqrt{n}\ln n}{\alpha^2}.$$
\end{proof}


\begin{thebibliography}{99}

\bibitem{AGG}
{ M. Akian, S. Gaubert, A. Guterman, } Linear independence over
tropical semirings and beyond, \textit{Contemporary Mathematics} 495 (2009) 1--38.

\bibitem{BH}
F. Babaee, J. Huh, A tropical approach to a generalized Hodge conjecture for positive currents, \textit{Duke Math. J.} 166 (2017) 2749--2813.

\bibitem{BCOQ}
F. Baccelli, G. Cohen, G.J. Olsder, J.P. Quadrat, \textit{Synchronization and Linearity}, Wiley, 1992.


\bibitem{CJR}
M. Chan, A. N. Jensen, E. Rubei, The 4x4 minors of a 5xn matrix are a tropical basis, \textit{Linear Algebra Appl.} 435 (2011) 1598--1611.

\bibitem{CR} J. E. Cohen, U. G. Rothblum, Nonnegative ranks, decompositions, and
factorizations of nonnegative matrices, \textit{Linear Algebra Appl.} 190 (1993) 149--168.

\bibitem{CDPR}
F. Cools, J. Draisma, S. Payne, E. Robeva, A tropical proof of the Brill-Noether theorem, \textit{Adv. Math.} 230 (2012) 759-776.

\bibitem{CG}
R. A. Cuninghame-Green, Minimax algebra, volume 166 of \textit{Lecture Notes in Economics
and Mathematical Systems}, Springer-Verlag, Berlin, 1979.

\bibitem{DSS}
M.~Develin, F.~Santos, B.~Sturmfels,On the rank of a tropical matrix, in \textit{Discrete and Computational Geometry} (E.~Goodman, J.~Pach and E.~Welzl, eds.), MSRI Publications, Cambridge Univ. Press, 2005.

\bibitem{FMPTdW}
S. Fiorini, S. Massar, S. Pokutta, H.R. Tiwary, R. de Wolf, Linear vs. semidefinite extended formulations: exponential separation and strong lower bounds, in \textit{Proc. 44th Symposium on Theory of Computation}, ACM, 2012.


\bibitem{Grun}
B. Gr\"{u}nbaum, \textit{Arrangements of hyperplanes}. Springer, New York, 2003.

\bibitem{GS}
A. Guterman, Ya. Shitov, Rank functions of tropical matrices, \textit{Linear Algebra Appl.} 498 (2016) 326--348.

\bibitem{Hoeff}
W. Hoeffding, Probability inequalities for sums of bounded random variables, \textit{J. Am. Stat. Assoc.} 58 (1963) 13--30.

\bibitem{Hru}
P. Hrube\v{s}, On the nonnegative rank of distance matrices, \textit{Inform. Process. Lett.} 112 (2012) 457--461.

\bibitem{Izh}
Z. Izhakian, Tropical arithmetic and matrix algebra, \textit{Commun. Algebra} 37 (2009) 1445--1468.

\bibitem{IRIR1}
Z. Izhakian, L. Rowen, Supertropical algebra, \textit{Adv. Math.} 225 (2010) 2222--2286.

\bibitem{IKR}
Z. Izhakian, M. Knebusch, L. Rowen, Supertropical linear algebra, \textit{Pacific J. Math.} 266 (2013) 43--75.

\bibitem{IKR2}
Z. Izhakian, M. Knebusch, L. Rowen, Layered tropical mathematics, \textit{J. Algebra} 416 (2014) 200--273.

\bibitem{IRN}
Z. Izhakian, A. Niv, L. Rowen, Supertropical $SL_n$, \textit{Linear Multilinear A.} (2017).

\bibitem{KR}
K. H. Kim, N. F. Roush, Kapranov rank vs. tropical rank, \textit{Proc. Amer. Math. Soc.} 134 (2006) 2487--2494.

\bibitem{KM}
V. N. Kolokoltsov, V. P. Maslov, \textit{Idempotent analysis and applications}, Kluwer Academic Publishers, 1997.

\bibitem{Kuhn}
H. W. Kuhn, The Hungarian Method for the assignment problem, \textit{Nav. Res. Logist. Q.} 2 (1955) 83--97.

\bibitem{LGAG}
Z. Li, Y. Gao, M. Arav, F. Gong, W. Gao, F. J. Hall, H. van der Holst, Sign patterns with minimum rank 2 and upper bounds on minimum ranks, \textit{Linear Multilinear A.} 61 (2013) 895--908.

\bibitem{McE}
W. M. McEneaney, Max-plus methods for nonlinear control and estimation, Systems $\verb"&"$ Control: Foundations $\verb"&"$ Applications, Birkh\"auser Boston Inc., Boston, MA, 2006.

\bibitem{Mikh}
G. Mikhalkin, Enumerative tropical algebraic geometry in $\R^2$, \textit{J. Amer. Math. Soc.} 18 (2005) 313--377.

\bibitem{Plus}
M. Plus. Linear systems in $(\max,+)$-algebra, in \textit{Proceedings of the 29th Conference on Decision and Control}, Honolulu, 1990.

\bibitem{PH}
P. L. Poplin, R. E. Hartwig, Determinantal identities over commutative semirings, \textit{Linear Algebra Appl.} 387 (2004) 99--132.

\bibitem{PP}
A. Padrol, J. Pfeifle, Polygons as Sections of Higher-Dimensional Polytopes, \textit{Electron. J. Comb.} 22 (2015) 1.24.


\bibitem{Sim}
I. Simon, Limited Subsets of a Free Monoid, in \textit{Proc. 19th Annual Symposium on Foundations of Computer Science}, Piscataway, N.J., Institute of Electrical and Electronics Engineers, 1978.



\bibitem{myarx1}
Ya. Shitov, Example of a 6-by-6 Matrix with Different Tropical and Kapranov Ranks, preprint (2010) arXiv:1012.5507.

\bibitem{mylaa}
Ya. Shitov, On the Kapranov ranks of tropical matrices, \textit{Linear Algebra Appl.} 436 (2012) 3247--3253.

\bibitem{mytrb}
Ya. Shitov, When do the r-by-r minors of a matrix form a tropical basis? \textit{J. Combin. Theory A} 120 (2013) 1166--1201.

\bibitem{myproc}
Ya. Shitov, Mixed subdivisions and ranks of tropical matrices, \textit{Proc. Amer. Math. Soc.} 142 (2014) 15--19.

\bibitem{my7x7}
Ya. Shitov, An upper bound for nonnegative rank, \textit{J. Comb. Theory A} 122 (2014) 126--132.

\bibitem{myspb}
Ya. Shitov, Tropical semimodules of dimension two, \textit{St. Petersb. Math. J.} 26 (2015) 341--350.


\bibitem{Vor}
N. N. Vorobyev, Extremal algebra of positive matrices, \textit{Elektron. Informationsverarbeitung und Kybernetik} 3 (1967) 39--71.

\end{thebibliography}
\end{document}